\documentclass[a4paper,12pt,reqno]{amsart} 

\usepackage[T1]{fontenc} 
\usepackage{amsmath} 
\usepackage{graphicx} 
\usepackage{xcolor} 
\usepackage{tikz}
\usepackage{subcaption}
\usetikzlibrary{calc, fit}
\usetikzlibrary{shapes.geometric}
\usepackage{amssymb}
\usepackage{amsthm}
\usetikzlibrary{cd}
\usepackage{euscript}
\usepackage{authblk}
\usepackage{euscript}
\usepackage{authblk}
\usepackage{amsaddr}
\usepackage{hyperref}
\usepackage{comment}

\theoremstyle{definition}
\newtheorem{lemma}{Lemma}
\newtheorem{theorem}{Theorem}
\newtheorem{proposition}[lemma]{Proposition}

\newtheorem{corollary}[lemma]{Corollary}

\newcommand{\m}[1]{{\bf {#1}}}									
\newcommand{\lgc}[1]{{\ensuremath{{{\mathbf{#1}}}}}}						
\newcommand{\cls}[1]{\ensuremath{{\sf #1}}}						
\newcommand{\opr}[1]{\ensuremath{{\mathbb #1}}}					
\newcommand{\prp}[1]{{\sf #1}}

\renewcommand{\phi}{\varphi}
\newcommand{\fb}{\varphi_\m{B}}
\newcommand{\fc}{\varphi_\m{C}}
\newcommand{\gb}{\psi_\m{B}}
\newcommand{\gc}{\psi_\m{C}}
\newcommand{\emd}{\hookrightarrow}

\DeclareMathOperator{\im}{im}

\newcommand{\eq}{\approx}

\newcommand{\Z}{\mathbb{Z}}

\newcommand{\var}{\mathsf{var}}

\newcommand{\hm}{\opr{H}}
\newcommand{\iso}{\opr{I}}
\newcommand{\pr}{\opr{P}}
\newcommand{\prd}{\opr{P}}
\newcommand{\pu}{\opr{P}_\mathsf{U}}
\newcommand{\sub}{\opr{S}}
\newcommand{\vr}{\opr{V}}
\newcommand{\quas}{\opr{Q}}

\newcommand{\ispu}{\iso\sub\pu}
\newcommand{\isppu}{\iso\sub\pr\pu}
\DeclareMathOperator{\hsp}{\hm\sub\prd}

\newcommand{\Q}{\cls{Q}}
\newcommand{\Qrfsi}{\cls{Q}_{\mathrm{RFSI}}}
\newcommand{\K}{\cls{K}}

\renewcommand{\L}{\cls{L}}

\newcommand{\SA}{\cls{SA}}
\newcommand{\SM}{\cls{SM}}

\newcommand{\omq}{\Omega}
\newcommand{\E}{\m{E}}
\renewcommand{\Z}{\m{Z}}

\newcommand{\cg}{\mathrm{Cg}_\m{A}^\Q}

\newcommand{\ut}{\textrm{\textup{e}}}
\newcommand{\tc}{\textrm{\textup{t}}}
\newcommand{\zr}{\textrm{\textup{f}}}
\newcommand{\meet}{\mathbin{\land}}
\newcommand{\join}{\mathbin{\lor}}

\newcommand{\fm}{\mathrm{Fm}}

\makeatletter
\g@addto@macro{\endabstract}{\@setabstract}
\newcommand{\authorfootnotes}{\renewcommand\thefootnote{\@fnsymbol\c@footnote}}%
\makeatother

\newcommand\extrafootertext[1]{%
    \bgroup
    \renewcommand\thefootnote{\fnsymbol{footnote}}%
    \renewcommand\thempfootnote{\fnsymbol{mpfootnote}}%
    \footnotetext[0]{#1}%
    \egroup
}

\begin{document}

%
%
\begin{center}
  \large
  {\bf Maehara Interpolation in Extensions of $\lgc{R}$-mingle} \par \bigskip

  \normalsize
  \authorfootnotes
  Wesley Fussner\footnote{e-mail: \url{fussner@cs.cas.cz}}\textsuperscript{1}, Krzysztof Krawczyk\footnote{e-mail: \url{krawczyk@cs.cas.cz}}\textsuperscript{1}\par  \bigskip

  \textsuperscript{1}Institute of Computer Science of the Czech Academy of Sciences \par
  \footnotetext[0]{This research was supported by the Czech Science Foundation project 25-18306M, INTERACT.}

\end{center}

\begin{abstract}
We show that there are exactly five quasivarieties of Sugihara algebras with the amalgamation property, and that all of these have the relative congruence extension property. As a consequence, we obtain that the amalgamation property and transferable injections property coincide for arbitrary quasivarieties of Sugihara algebras. These results provide a complete description of arbitrary (not merely axiomatic) extensions of the logic $\lgc{R}$-mingle that have the Maehara interpolation property, and further demonstrates that the Robinson property and Maehara interpolation property coincide for arbitrary extensions of $\lgc{R}$-mingle. Further, we show that the question of whether a given finitely based extension of $\lgc{R}$-mingle has the Maehara interpolation property is decidable.
\end{abstract}

\section{Introduction}
\label{sec:intro}

Relevant logics have long been the site of intriguing work on interpolation. Most famously, Urquhart gave an elegant geometric proof in \cite{Urq93} of the failure the Craig interpolation property in the logic of relevant implication $\lgc{R}$, as well as a range of other systems in the relevant family, such as the logic $\lgc{E}$ of entailment and $\lgc{T}$ of ticket entailment. He subsequently gave simpler, more direct proofs in \cite{Urq99} and \cite{RRR2}. These studies comprise a celebrated and well known line of research, and have sometimes been understood by logicians as the final word on this topic or, even worse, as saying `interpolation fails in relevant logic' simpliciter. However, there is in fact much more to be said about interpolation for systems not covered by Urquhart's arguments, some of which \emph{do} enjoy sundry interpolation properties. Indeed, already at the dawn of the subject, Anderson and Belnap showed at \cite[p.~161]{AB75} that the logic of first-degree entailment has the so-called \emph{perfect Craig interpolation property}, that is:
\begin{equation}
  \tag{\prp{PCIP}}\label{eq:PCIP}
  \parbox{\dimexpr\linewidth-6em}{%
    \strut
	If $\vdash\alpha\to\beta$, then there is a formula $\delta$ such that $\var(\delta)\subseteq\var(\alpha)\cap\var(\beta)$ and 	both $\vdash\alpha\to\delta$ and $\vdash\delta\to\beta$.
    \strut
    }
\end{equation}
On the other hand, they also show at \cite[pp.~416-417]{AB75} that the logic $\lgc{RM}$---that is, $\lgc{R}$ with the addition of the mingle axiom $\alpha\to(\alpha\to\alpha)$---lacks the \emph{Craig interpolation property} (\emph{\prp{CIP}}) in the following familiar (imperfect) form:
\begin{equation}
  \tag{\prp{CIP}}\label{eq:CIP}
  \parbox{\dimexpr\linewidth-6em}{%
    \strut
	If ${\not\vdash}~\neg\alpha$ , ${\not\vdash}~\beta$, and $\vdash\alpha\to\beta$, then there is a formula $\delta$ such 	that $\var(\delta)\subseteq\var(\alpha)\cap\var(\beta)$ and both $\vdash\alpha\to\delta$ and $\vdash\delta\to\beta$.
    \strut
    }
\end{equation}
The logic $\lgc{RM}$ is here formulated without the truth constants $\tc$ and $\zr$, and this turns out to make a big difference: Meyer showed in \cite{MeyerReport} that the logic $\lgc{RM}^\tc$---that is, $\lgc{RM}$ with $\tc$ and $\zr$---has the \prp{CIP}.\footnote{Or at least this has been reliably reported in the literature. The cited technical report of Meyer appears to be lost to time and we could not procure a copy. The authors do not know, and could not gather from discussions with firsthand witnesses, the precise form of interpolation that Meyer established in \cite{MeyerReport}.}
 As shown in \cite{KiharaOno2010,GalatosJipsenKowalskiOno2007}, for $\lgc{RM}^\tc$ and other substructural logics validating the exchange rule, the perfect Craig interpolation property entails the following \emph{deductive interpolation property} (\emph{\prp{DIP}}):
\begin{equation}
  \tag{\prp{DIP}}\label{eq:DIP}
  \parbox{\dimexpr\linewidth-6em}{%
    \strut
	If $\alpha\vdash\beta$, then there is a formula $\delta$ such that $\var(\delta)\subseteq\var(\alpha)\cap\var(\beta)$ and 		both $\alpha\vdash\delta$ and $\delta\vdash\beta$.
    \strut
    }
\end{equation}
The converse is \emph{not} true: It is shown in \cite{FSLinear} that there are continuum-many substructural logics with exchange that have the \prp{DIP} but not the \prp{PCIP}. However, Marchioni and Metcalfe show in \cite{MarMet2012} that the \prp{DIP}, as formulated above, coincides with the \prp{PCIP} for axiomatic extensions of $\lgc{RM}^\tc$, and prove that there are just nine axiomatic extensions of $\lgc{RM}^\tc$ that have these two equivalent properties.

Taking a wider viewpoint, one salient feature of $\lgc{RM}^\tc$ is that it is \emph{semilinear}: It is characterized by linearly ordered algebraic models. Semilinear logics have themselves attracted quite a lot of attention, particularly among fuzzy logicians, who often take semilinearity to be characteristic of fuzziness (see, e.g., \cite{Hajek1998}). They also have been an important source of insight contributing to our current understanding of interpolation in substructural logics writ large (see, e.g., \cite{FSsemilin,FG1,FG2,FMS2024,GLT15,FSBasic}). Notably, the models of $\lgc{RM}^\tc$ do not exhaust all linearly ordered models of $\lgc{R}^\tc$, and one may quite profitably investigate the semilinear extensions of $\lgc{R}$ and $\lgc{R}^\tc$ themselves. Indeed, the extension of $\lgc{R}^\tc$ characterized by linearly ordered algebraic models does not itself have the \prp{DIP}, but nevertheless there are infinitely many semilinear extensions of $\lgc{R}^\tc$ that \emph{do} have the \prp{DIP}; see \cite[Proposition~5.3]{FSsemilin}.

The present paper returns to the question of interpolation in $\lgc{RM}$ (without constants). While it was already known early on in the history of relevant logic that $\lgc{RM}$ lacks the \prp{CIP}, variants of interpolation focusing on consequence rather than implication appear to have not been studied in the absence of truth constants. In this study, we focus on the following strong form of the \prp{DIP}---variously called the \emph{strong deductive interpolation property} or the \emph{Maehara interpolation property} (or \emph{\prp{MIP}} for short)---here tailored for logics omitting truth constants:
\begin{equation}
  \tag{\prp{MIP}}\label{eq:MIP}
  \parbox{\dimexpr\linewidth-6em}{%
    \strut
	If $\var(\Sigma\cup\{\alpha\})\cap\var(\Gamma)\neq\emptyset$ and $\Sigma,\Gamma\vdash\alpha$, there exists a set of 		formulas $\Delta$ such that $\var(\Delta)\subseteq\var(\Sigma\cup\{\alpha\})\cap\var(\Gamma)$ and both 
	$\Gamma\vdash\Delta$ and $\Sigma,\Delta\vdash\alpha$.
    }
\end{equation}
Like many systematic studies of interpolation in nonclassical logics, our methodology centers on algebraic semantics and, in particular, here focuses on \emph{Sugihara algebras}, which make up the equivalent algebraic semantics of $\lgc{RM}$. Consequently, our contributions have a two-part nature, reflecting both the algebraic and logical manifestations of the results. 

These contributions are several. First, after laying down some background material in Section~\ref{sec:sugihara}, we obtain, in Theorem~\ref{lem:AP classification}, a complete description of quasivarieties of Sugihara algebras with the amalgamation property. Our means of getting to this classification is through the by-now-standard methodology of \emph{closure properties} (cf.~\cite{FSBasic,FMS2024}): We prove a number of lemmas of the form `if $\Q$ is a quasivariety of Sugihara algebras with the amalgamation property and $\m{A}\in\Q$, then $\m{B}\in\Q$ as well', and then show that such closure lemmas suffice to specify all quasivarieties of Sugihara algebras with amalgamation. This is, as far as the authors are aware, the first successful application of this methodology to classify the subquasivarieties of a given quasivariety with amalgamation, as opposed to merely subvarieties. One of the main difficulties of applying this strategy in the setting of quasivarieties is that the relative congruence extension property---unlike the absolute congruence extension property considered in the setting of varieties---does not transfer from a quasivariety to its subquasivarieties. The fact that most quasivarieties of Sugihara algebras lack the relative congruence extension property (see \cite{Czel1999}) is one of the chief obstacles that we overcome here, hewing to proofs that avoid the congruence extension property throughout our development.

As a by-product of eschewing the relative congruence extension property in our arguments, we obtain the surprising result (Corollary~\ref{cor:from AP to RCEP}) that, for quasivarieties of Sugihara algebras, the amalgamation property implies the relative congruence extension property. This is especially interesting in light of Kearnes' result that for a residually small, congruence-distributive variety, the amalgamation property implies the congruence extension property \cite{Kearnes1989}. As a further corollary, we also obtain, in Corollary~\ref{cor:from AP to TIP}, that the amalgamation property coincides with the transferable injections property for arbitrary quasivarieties of Sugihara algebras. 

Because quasivarieties of Sugihara algebras give equivalent algebraic semantics for extensions of $\lgc{RM}$, the amalgamation property for a given quasivariety $\Q$ of Sugihara algebras corresponds directly with the Robinson property for the extension $\lgc{L}$ of $\lgc{RM}$ corresponding to $\Q$, and likewise the transferable injections property for $\Q$ corresponds to the Maehara interpolation property for $\lgc{L}$. Thus we obtain, in Theorem~\ref{thm:MIP classification}, a complete list of extensions of $\lgc{RM}$ with Maehara interpolation and, moreover, give explicit finite quasiequational bases for each of these. Further, as a consequence of Corollary~\ref{cor:from AP to TIP}, we obtain (Proposition~\ref{prop:RP iff MIP}) that the Robinson property and Maehara interpolation property coincide for any extension of $\lgc{RM}$. This can be compared to the coincidence of the \prp{CIP} and \prp{DIP} for extensions of $\lgc{RM}^\tc$ (see \cite{MarMet2012}). Finally, using the fact that Sugihara algebras are locally finite and the classification in Theorem~\ref{thm:MIP classification}, we obtain, in Theorem~\ref{thm:decidability}, that it is effectively decidable whether a finitely based quasivariety of Sugihara algebras has the amalgamation property and, likewise, whether an extension of $\lgc{RM}$ by finitely many rules has the \prp{MIP}.

\section{Algebraic preliminaries and the basics of Sugihara algebras}\label{sec:sugihara}
Before beginning our main argument, we summarize some key definitions and facts that we will invoke later on.

\subsection{Algebra and amalgamation}

First, we lay out a few of the most important definitions from universal algebra. Let $\K$ and $\L$ be classes algebras in a common similarity type with $\K\subseteq\L$. A \emph{half span in $\K$} is a pair $\langle\fb\colon\m{A}\emd\m{B},\fc\colon\m{A}\to\m{C}\rangle$ of homomorphisms, where $\m{A},\m{B},\m{C}\in\K$ and $\fb$ is an embedding. If $\langle\fb\colon\m{A}\emd\m{B},\fc\colon\m{A}\emd\m{C}\rangle$ is a half span in $\K$ such that $\fc$ is also an embedding, then we say that $\langle\fb,\fc\rangle$ is a \emph{span in $\K$}. Given a span $\langle\fb\colon\m{A}\emd\m{B},\fc\colon\m{A}\emd\m{C}\rangle$ in $\K$, an \emph{amalgam of $\langle\fb,\fc\rangle$ in $\L$} is a pair $\langle\gb\colon\m{B}\emd\m{D},\gc\colon\m{C}\emd\m{D}\rangle$ of embeddings, where $\m{D}\in\L$ and $\gb\fb=\gc\fc$.  A class $\K$ of similar algebras is said to have the \emph{amalgamation property} (or \emph{\prp{AP}}) if each span in $\K$ has an amalgam in $\K$. If $\Q$ is any quasivariety and $\Lambda(\Q)$ is its lattice of subquasivarieties, we denote by $\omq(\Q)$ the subposet of $\Lambda(\Q)$ consisting of all $\L\in\Lambda(\Q)$ such that $\L$ has the \prp{AP}. See the lefthand side of Figure~\ref{fig:AP and TIP}.

If $\K$ is any class of algebras in a common similarity type, we say that $\K$ has the \emph{transferable injections property} (or \emph{\prp{TIP}}) if for any half span $\langle\fb\colon\m{A}\emd\m{B},\fc\colon\m{A}\to\m{C}\rangle$ in $\K$, there exists a pair $\langle\gb\colon\m{B}\to\m{D},\gc\colon\m{C}\emd\m{D}\rangle$ of homomorphisms such $\m{D}\in\K$, $\gc$ is an embedding, and $\gb\fb=\gc\fc$. See the righthand side of Figure~\ref{fig:AP and TIP}.

\begin{figure}[t]
\begin{center}
\captionsetup{justification=centering}
\begin{tabular}{ccc}
\begin{tikzcd}
	& {\bf B} \\
	{\bf A} && {\bf D} \\
	& {\bf C}
	\arrow["{\fb}", hook, from=2-1, to=1-2]
	\arrow["{\gb}", dashed, hook, from=1-2, to=2-3]
	\arrow["{\fc}"', hook, from=2-1, to=3-2]
	\arrow["{\gc}"', dashed, hook, from=3-2, to=2-3]
\end{tikzcd}\hspace{0.3 in}
& &
\begin{tikzcd}
	& {\bf B} \\
	{\bf A} && {\bf D} \\
	& {\bf C}
	\arrow["{\fb}", from=2-1, to=1-2]
	\arrow["{\gb}", dashed, hook, from=1-2, to=2-3]
	\arrow["{\fc}"', hook, from=2-1, to=3-2]
	\arrow["{\gc}"', dashed,  from=3-2, to=2-3]
\end{tikzcd}
\\ \\
(\prp{AP})\hspace{0.31 in} & & (\prp{TIP})
\end{tabular}
\end{center}
\caption{Commutative diagrams for the \prp{AP} (left) and \prp{TIP} (right).}
\label{fig:AP and TIP}
\end{figure}
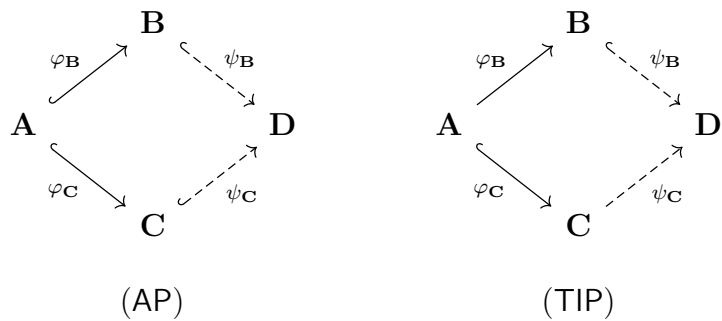

For a class of similar algebras $\K$, we denote by $\hm(\K)$, $\iso(\K)$, $\sub(\K)$, $\pr(\K)$, and $\pu(\K)$ the closure of $\K$ under homomorphic images, isomorphisms, subalgebras, products, and ultraproducts, respectively. We further denote by $\vr(\K)=\hsp(\K)$ the variety generated by $\K$, and by $\quas(\K) = \isppu(\K)$ the quasivariety generated by $\K$. For any class $\K$ of algebras, we denote the class of finitely generated members of $\K$ by $\K_\mathrm{FG}$.

Now let $\Q$ be any quasivariety and $\m{A}\in\Q$. A congruence $\Theta$ of $\m{A}$ is called a \emph{$\Q$-congruence}, or simply a \emph{relative congruence} if $\Q$ is understood, provided that $\m{A}/\Theta\in\Q$. A quasivariety $\Q$ is said to have the \emph{relative congruence extension property} (or \emph{\prp{RCEP}}) if whenever $\m{A},\m{B}\in\Q$, $\m{A}$ is a subalgebra of $\m{B}$, and $\Theta$ is a relative congruence of $\m{A}$, there exists a relative congruence $\Psi$ of $\m{B}$ such that $\Theta = \Psi\cap A^2$. When $\Q$ is a variety, every congruence of an algebra in $\Q$ is a $\Q$-congruence and, if $\Q$ has the \prp{RCEP}, we just say that $\Q$ has the \emph{congruence extension property} (or \emph{\prp{CEP}}).

An algebra $\m{A}$ is called \emph{directly indecomposable} if it is not isomorphic to $\m{B}\times\m{C}$ for any nontrivial algebras $\m{B}$, $\m{C}$.

Given any quasivariety $\Q$, an algebra in $\m{A}\in\Q$ is called \emph{relatively (finitely) subdirectly irreducible} if whenever $\m{A}$ is a (finite) subdirect product of algebras in $\Q$, then $\m{A}$ is isomorphic to one of those algebras. Note that $\m{A}\in\Q$ is relatively subdirectly irreducible if the equality congruence $\Delta_{\m{A}} = \{(a,a)\mid a\in A\}$ is completely meet irreducible in the lattice of relative congruences of $\m{A}$, and $\m{A}$ is relatively finitely subdirectly irreducible if $\Delta_{\m{A}}$ is meet irreducible in the lattice of relative congruences.\footnote{It is a consequence of this definition that the trivial algebra is relatively finitely subdirectly irreducible.} We denote the classes of relatively subdirectly irreducible and relatively finitely subdirectly irreducible algebras in a class $\K\subseteq\Q$ by $\K_\mathrm{RSI}$ and $\K_\mathrm{RFSI}$, respectively. Here, of course, we take care to only use this (a priori ambiguous) notation when $\Q$ is apparent from context. When $\Q$ is a variety, we needn't relativize these definitions and they instead reduce to the usual notions of subdirect irreducibility and finite subdirect irreducibility. In this case, we drop the `R' from our notation and write $\K_\mathrm{SI}$ and $\K_\mathrm{FSI}$ for the classes of subdirectly irreducible and finitely subdirectly irreducible members of $\K$, respectively. 

The following result is often called the \emph{Relativized J\'{o}nsson Lemma} and it is fundamental to working with lattices of subquasivarieties.

\begin{lemma}[see, e.g., {\cite[Lemma~1.5]{Czel1990}}]
Let $\K$ be any class of similar algebras. Then every nontrivial member of $\quas(\K)_\mathrm{RFSI}$ belongs to $\ispu(\K)$.
\end{lemma}

For any quasivariety $\Q$, any $\m{A}\in\Q$, and any $R\subseteq A^2$, we denote the least relative congruence of $\m{A}$ containing $R$ by $\cg(R)$. A given quasivariety $\Q$ is said to have \emph{equationally definable principal meets} (or \emph{\prp{EDPM}}) if there exist finitely many pairs of four-variable terms $(p_i(x,y,z,w),q_i(x,y,z,u))$, $i\in\{1,\ldots,n\}$, such that for any $\m{A}\in\Q$ and any elements $a,b,c,d\in A$,
\[
\cg(a,b)\cap\cg(c,d) = \cg(\{(p_i(a,b,c,d),q_i(a,b,c,d) \mid 1\leq i\leq n\}).
\]
The next two lemmas will be used in tandem in the proof of the key Lemma~\ref{lem:with the ap}.

\begin{lemma}[{\cite[Corollary~2.6]{Czel1990}}]\label{lem:eneveloping variety}
Let $\Q$ be any quasivariety such that $\vr(\Q)$ is congruence distributive. Then $\Q$ has \prp{EDPM} if and only if $\Q=\quas(\K)$ for some class $\K$ such that $\ispu(\K)\subseteq\vr(\K)_\mathrm{FSI}$.
\end{lemma}

\begin{lemma}[{\cite[Theorem~2.3]{Czel1990}}]\label{lem:czel universal}
Let $\Q$ be a quasivariety with \prp{EDPM}. Then $\Q$ is congruence distributive and $\Qrfsi$ forms a universal class.
\end{lemma}

We conclude our treatment of general algebraic preliminaries with the following useful equivalences for testing the \prp{AP} and \prp{TIP}.

\begin{lemma}[{\cite[Theorem~3.4]{Fussner2024}}]\label{lem:transfer}
Let $\Q$ be any quasivariety with the \prp{RCEP} such that $\Qrfsi$ is closed under subalgebras. Then $\Q$ has the \prp{AP} if and only if every span of algebras in $\Q_\mathrm{FG}\cap\Qrfsi$ has an amalgam in $\Q$.
\end{lemma}

\begin{lemma}[{\cite[Proposition~4.1]{Fussner2024}}]\label{lem:tip}
Let $\Q$ be any quasivariety. Then $\Q$ has the \prp{TIP} if and only if $\Q$ has the \prp{RCEP} and the \prp{AP}.
\end{lemma}

\subsection{Sugihara algebras and Sugihara monoids}
We now turn to the classes of algebras that will be of primary interest to our study. First, a \emph{commutative involutive residuated lattice} is an algebra $\m{A} = \langle A,\meet,\join,\cdot,\to,\neg,\ut\rangle$ such that:
\begin{itemize}
\item $\langle A,\meet,\join\rangle$ is a lattice.
\item $\langle A,\cdot,\ut\rangle$ is a commutative monoid.
\item For all $x,y,z\in A$,
\[
x\cdot y\leq z \iff x\leq y\to z.
\]
\item For all $x,y\in A$, $x\to\neg y = y\to\neg x$.
\end{itemize}
In commutative involutive residuated lattices, we most often write $x\cdot y$ as simply $xy$. 

A commutative involutive residuated lattice is called a \emph{Sugihara monoid} provided that multiplication $\cdot$ is idempotent---that is, $x^2\eq x$ holds---and the lattice $\langle A,\meet,\join\rangle$ is distributive. A Sugihara monoid is called \emph{odd} if it satisfies $\neg\ut \approx\ut$. \emph{Sugihara algebras} are the $\ut$-free subreducts of Sugihara monoids. Each of the classes of Sugihara monoids and Sugihara algebras comprises a variety with the \prp{CEP}. We denote these varieties by $\SM$ and $\SA$, respectively. 

We define an algebra $\Z = \langle\mathbb{Z},\meet,\join,\cdot,\to,\neg\rangle$, where $\meet$ and $\join$ are the operations of binary minimum and maximum with respect to the usual order on $\mathbf{Z}$, $\neg x = -x$ is additive inversion, and $\cdot$ and $\to$ are given by the following formulas:

\[ x\cdot y = \begin{cases} 
      x & |x| > |y| \\
      y & |x| < |y| \\
      x\meet y & |x|=|y|, 
   \end{cases}
\]
\[ x\to y = \begin{cases} 
      (-x)\join y & x\leq y \\
      (-x)\meet y & x\not\leq y.
   \end{cases}
\]
Here $|x|=x\to x$ is the usual absolute value function defined on the integers, and one may readily check that the multiplication given above is the infimum with respect to the non-standard order on $\mathbb{Z}$ given by
\[
\cdots < -3 < 3 < -2 < 2 < -1 < 1 < 0.
\]
The algebra $\Z$ is a Sugihara algebra, and it may be expanded to a Sugihara monoid $\Z^0$ by including a constant designating the multiplicative unit $0$. Indeed, one may show that $\SA = \vr(\Z)$ and $\SM = \vr(\Z^0)$. Thus, both of $\SA$ and $\SM$ are semilinear in the sense that they are generated as varieties by totally ordered algebras.

We will denote the subalgebra of $\Z$ with universe $E=\mathbb{Z}-\{0\}$ by $\E$. Further, for each integer $n\geq 0$, each of the set $\{-2n-1,-2n,\ldots,-1,0,1,\ldots,2n,2n+1\}$ gives the universe of a algebra of $\Z$ that we denote $\Z_{2n+1}$, and, for $n\geq 1$, each of the sets $\{-2n-1,-2n,\ldots,-1,1,\ldots,2n,2n+1\}$ gives the universe of a subalgebra $\Z_{2n}$ of $\E$. It turns out that the lattice of subvarieties of $\SA$ forms a countable chain given by
\[
\vr(\Z_1)\subseteq\vr(\Z_2)\subseteq\vr(\Z_3)\subseteq\cdots\vr(\Z)=\vr(\E)=\SA.
\]
Further, $\SA = \quas(\Z) = \quas(\{\Z_n \mid n\geq 1\}) = \quas(\{\Z_{2n+1}\mid n\geq 0\})$, and also $\quas(\E) = \quas(\{\Z_{2n} \mid n\geq 1\})$ is a proper subquasivariety of $\SA$.

An element $x$ in a Sugihara algebra is called \emph{positive} if $|x|=x$. The (finitely) subdirectly irreducible and directly indecomposable, and we will use these classifications extensively.

\begin{lemma}[see, e.g., {\cite{FontPerez1992}}]\label{lem:totally ordered}
\begin{enumerate}
\item[]
\item A Sugihara algebra is finitely subdirectly irreducible if and only if it is totally ordered.
\item A nontrivial Sugihara algebra is subdirectly irreducible if and only if it is totally ordered, and additionally there is $u\in A$ such that $x\in A$ is positive if and only if $u< x$.
\item Every totally ordered Sugihara algebra of cardinality $n$ is isomorphic to $\Z_n$.
\end{enumerate}
\end{lemma}

Note that Sugihara algebras are locally finite by \cite[Lemma~1.1(1)]{Blok1986}, so the finitely generated Sugihara algebras are precisely the finite ones. In particular, according to the preceding lemma, $\m{A}$ is a (nontrivial) finitely generated subdirectly irreducible Sugihara algebra if and only if $\m{A}$ is isomorphic to $\Z_n$ for some $n\geq 2$.

If $\m{A}$ is any Sugihara algebra, we will denote by $\bot\m{A}\top$ the nested sum $\Z_3\oplus\m{A}$ (see \cite{FMS2024,FG1,FG2}). That is, $\bot\m{A}\top$ is the algebra obtained from $\m{A}$ by adjoining $\bot$ as a new least element, $\top$ as a new greatest element, and defining $\neg a=\neg^\m{A} a$ for $a\in\mathsf{A}$, $\neg\top=\bot$, $\neg\bot=\top$, and
\[
a\cdot b = 
\begin{cases}
a\cdot^\m{A} b & \text{if }a,b\in A,\\
\bot & \text{if }a=\bot\text{ or }b=\bot,\\
\top & \text{otherwise},
\end{cases}
\]
\begin{align*}
a\to b = 
\begin{cases}
a\to^\m{A} b & \text{if }a,b\in\mathsf{A},\\
\top & \text{if }a=\bot\text{ or }b=\top,\\
\bot & \text{otherwise}.
\end{cases}
\end{align*}
We also iterate this construction, writing $\bot^{n+1}\m{A}\top^{n+1}$ to indicate the algebra $\bot\bot^n\m{A}\top^n\top$.

\begin{lemma}[{\cite[Corollary~2.5]{Blok1986}}]\label{lem: directly indecomposable}
A finite Sugihara algebra $\m{A}$ is directly indecomposable if and only if $\m{A}\in\iso(\Z_2)$ or $\m{A}\in\iso(\bot\m{B}\top)$, where $\m{B}$ is a finite Sugihara algebra.
\end{lemma}

The following technical lemma is useful in performing computations within Sugihara algebras. The properties announced therein are familiar to most specialists, and we will use them without specific reference.

\begin{lemma}[see \cite{FGSugihara,CabPri2020}]\label{lem:basic props}
Let $\m{A}\in\SA$ and $x,y,z\in A$. Then:
\begin{enumerate}
\item $x(y\join z) = xy\join xz$.
\item $x(y\meet z) = xy\meet xz$.
\item $x (x\to y) \leq y$.
\item If $x\leq y$, then $xz\leq yz$, $z\to x\leq z\to y$, and $y\to z\leq x\to z$.
\item $x\to (y\join z) = (x\to y)\join (x\to z)$.
\item $(x\meet y)\to z = (x\to z)\join (y\to z)$.
\item $x\to (y\meet z) = (x\to y)\meet (x\to z)$.
\item $(x\join y)\to z = (x\to z)\meet (y\to z)$.
\item $x\cdot\neg x = x\meet\neg x$.
\item $\neg\neg x = x$.
\item $\neg (x\join y) = \neg x\meet\neg y$.
\item $\neg (x\meet y) = \neg x \join\neg y$.
\end{enumerate}
\end{lemma}

The next lemma is sometimes useful in deriving information about Sugihara algebras from Sugihara monoids.

\begin{lemma}\label{lem:generation fixed point}
Let $\m{A}$ be a totally ordered Sugihara algebra and let $\ut$ be an involution fixed point in $\m{A}$. Then $A-\{\ut\}$ is the universe of a subalgebra of $\m{A}$.
\end{lemma}

\begin{proof}
It follows from the definitions of the basic operations that
\[
x\meet y,x\join y,x\to y,\neg x\in \{x,y,\neg x,\neg y\}
\]
for any $x,y\in A$. If neither $x$ nor $y$ is an involution fixed point, then it follows that none of $x\meet y$, $x\join y$, $x\to y$, or $\neg x$ is an involution fixed point.
\end{proof}

Note that each finite totally ordered Sugihara algebra may be embedded (as a Sugihara algebra) in a totally ordered odd Sugihara monoid. For $\Z_{2n+1}$, this may be done by designating the unique negation-fixed element as a constant. For $\Z_{2n}$, this follows by embedding in $\Z_{2n+1}$, considered as an odd Sugihara monoid.

\begin{lemma}[{\cite[Theorem~3.5]{Krawczyk2022}}]\label{lem:split}
Let $\Q$ be any quasivariety of Sugihara algebras properly containing $\vr (\Z_2)$. Then either $\Z_2\times\Z_3\in\Q$ or else $\Z_2\times\Z_4\in\Q$.
\end{lemma}

\begin{lemma}\label{lem:splitting for v3}
Let $\Q$ be a quasivariety of Sugihara algebras that properly contains $\vr(\Z_3)$. Then $\Z_2\times\Z_4\in\Q$.
\end{lemma}

\begin{proof}
Since $\Q$ is a locally finite quasivariety, it is generated by its finite members. Consequently, since $\Q$ contains $\vr(\Z_3)$ properly, there exists some finite algebra $\m{A}\in\Q$ with $\m{A}\notin\vr(\Z_3)$. As $\m{A}$ is finite, it may be written as a direct product $\m{A}_1\times\cdots\times\m{A}_n$ of finitely many directly indecomposable Sugihara algebras $\m{A}_1,\ldots,\m{A}_n$. Because $\m{A}\notin\vr(\Z_3)$, there exists some $j\in\{1,\ldots,n\}$ such that $\m{A}_j\notin\vr(\Z_3)$. Noting that every nontrivial Sugihara algebra contains a subalgebra isomorphic to $\Z_2$ and every directly indecomposable Sugihara algebra not in $\vr(\Z_3)$ contains a subalgebra isomorphic to $\Z_4$, we have that
\[
\Z_2\in\iso\sub\left(\prod_{\substack{i\neq j \\ j=1,\ldots,n}}\m{A}_i\right)\text{ and }\Z_4\in\iso\sub(\Z_j).
\]
Because $\prd\sub(\K)\subseteq\sub\prd(\K)$ for any class $\K$ of similar algebras, it follows that $\Z_2\times\Z_4$ embeds in
\[
\left(\prod_{\substack{i\neq j \\ j=1,\ldots,n}}\m{A}_i\right)\times\m{A}_j\cong\m{A},
\]
and, as $\Q$ is closed under $\iso$ and $\sub$, the result follows.
\end{proof}

\begin{lemma}[{\cite[Lemma~6.3]{GilFerez2020}}]\label{lem:osm}
The class of totally ordered odd Sugihara monoids has the amalgamation property.
\end{lemma}
\begin{lemma}[{\cite[Theorem~2.1]{Czel1999}}]\label{lem:RCEP for Q SA}
Let $\Q$ be a quasivariety of Sugihara algebras. Then $\Q$ has the \prp{RCEP} if and only if $\Q$ is one of $\vr(\Z)$, $\quas(\E)$, $\vr(\Z_n)$, or $\quas(\Z_{2n})$ for some $n\geq 1$.
\end{lemma}

\section{Amalgamation in Sugihara algebras}
\label{sec:positive}

In this section, we will give an exhaustive classification of arbitrary quasivarieties of Sugihara algebras with the \prp{AP}, as well as a concrete description of $\omq(\SA)$, the poset of quasivarieties of Sugihara algebras with the \prp{AP}. At the outset, we will obtain the relatively straightforward, affirmative result that the four quasivarieties $\vr(\Z_2)$, $\vr(\Z_3)$, $\vr(\Z)$, and $\quas(\E)$ have the \prp{AP}. The bulk of our work will be directed toward showing that, along with the trivial quasivariety, these are \emph{all} the subquasivarieties of Sugihara algebras with the \prp{AP}. In service toward the latter goal, we will, as previously mentioned, deploy a method based on closure properties as in, e.g., \cite{FSBasic} and \cite[Section~5]{FMS2024}.

We begin with our affirmative claim.

\begin{lemma}\label{lem:with the ap}
Each of the quasivarieties $\vr(\Z_2)$, $\vr(\Z_3)$, $\vr(\Z)$, and $\quas(\E)$ has the amalgamation property.
\end{lemma}

\begin{proof}
Note that each of $\vr(\Z_2)$, $\vr(\Z_3)$, $\vr(\Z)$, and $\quas(\E)$ has the \prp{RCEP} by Lemma~\ref{lem:RCEP for Q SA}. Further, since each of $\vr(\Z_2)$, $\vr(\Z_3)$, and $\vr(\Z)=\vr(\E)$ has the \prp{CEP}, Lemma~\ref{lem:eneveloping variety} gives that each of $\vr(\Z_2)$, $\vr(\Z_3)$, $\vr(\Z)$, and $\quas(\E)$ has \prp{EDPM}. Consequently, by Lemma~\ref{lem:czel universal}, each of $\vr(\Z_2)_\mathrm{FSI}$, $\vr(\Z_3)_\mathrm{FSI}$, $\vr(\Z)_\mathrm{FSI}$, and $\quas(\E)_\mathrm{RFSI}$ forms a universal class and thus each of these is, in particular, closed under taking subalgebras. Therefore, by Lemma~\ref{lem:transfer}, it suffices to show that, in each case, spans of finitely generated relatively finitely subdirectly irreducibles may be amalgamated in the respective quasivariety.

By the Relativized J\'{o}nsson Lemma, the finitely generated relatively subdirectly irreducible members of the given quasivarieties are
\begin{align*}
\vr(\Z_2)_\mathrm{FG}\cap\vr(\Z_2)_\mathrm{RFSI} &= \iso(\{\Z_1,\Z_2\}),\\
\vr(\Z_3)_\mathrm{FG}\cap\vr(\Z_3)_\mathrm{RFSI} &= \iso(\{\Z_1,\Z_2,\Z_3\}),\\
\vr(\Z)_\mathrm{FG}\cap\vr(\Z)_\mathrm{RFSI} &= \iso(\{\Z_n \mid n\geq 1\}),\\
\quas(\E)_\mathrm{FG}\cap\quas(\E)_\mathrm{RFSI} &= \iso(\{\Z_{2n} \mid n\geq 1\}).
\end{align*}
It is routine to verify that every span in $\iso(\{\Z_1,\Z_2\})$ has an amalgam in $\{\Z_1,\Z_2\}$, and every span in $\iso(\{\Z_1,\Z_2,\Z_3\})$ has an amalgam in $\{\Z_1,\Z_3\}$. Hence, both $\vr(\Z_2)$ and $\vr(\Z_3)$ have the \prp{AP}.

Now suppose $\langle \alpha\colon\m{A}\to\m{B},\beta\colon\m{A}\to\m{C}\rangle$ is a span in $\quas(\E)_\mathrm{FG}\cap\quas(\E)_\mathrm{RFSI} = \iso(\{\Z_{2n} \mid n\geq 1\})$. Then the embeddings $\alpha$ and $\beta$ may be uniquely extended to embeddings $\hat{\alpha}\colon\hat{\m{A}}\to\hat{\m{B}}$ and $\hat{\beta}\colon\hat{\m{A}}\to\hat{\m{C}}$ of totally ordered odd Sugihara monoids. The resulting span $\langle\hat{\alpha},\hat{\beta}\rangle$ has an amalgam $\langle\hat{\alpha}'\colon\hat{\m{B}}\to\hat{\m{D}},\hat{\beta}'\colon\hat{\m{C}}\to\hat{\m{D}}\rangle$ among totally ordered odd Sugihara monoids by Lemma~\ref{lem:osm}. Let $\m{D}$ be the subalgebra of the $0$-free reduct $\hat{\m{D}}$ generated by $\hat{\alpha}'[B]\cup\hat{\beta}'[C]$, where subalgebra generation is performed as a Sugihara algebra. Then $\m{D}$ is a finite Sugihara algebra chain since $\hat{\alpha}'[B]\cup\hat{\beta}'[C]$ is a finite set and Sugihara algebras are locally finite. Further, $\m{D}$ does not contain an involution fixed point by Lemma~\ref{lem:generation fixed point}, so $\m{D}$ is isomorphic to $\Z_{2n}$ for some $n\geq 1$ and, in particular, $\m{D}\in\quas (\E)$. Letting $\alpha'$ be the restriction of $\hat{\alpha}'$ to $B$ and $\beta'$ be the restriction of $\hat{\beta}'$ to $C$, we obtain that $\langle\alpha'\colon\m{B}\to\m{D},\beta'\colon\m{C}\to\m{D}\rangle$ is an amalgam of $\langle\alpha,\beta\rangle$ in $\quas (\E)$. Thus, $\quas (\E)$ has the \prp{AP}.

The proof that $\vr (\Z)$ has the \prp{AP} is similar. 
\end{proof}
Thus, including the trivial quasivariety, we have identified five quasivarieties of Sugihara algebras with the \prp{AP}. To show that these are all of them, we will use the closure properties exhibited in the next four lemmas. 

\begin{lemma}\label{lem:3reduction}
Let $\Q\in\omq(\SA)$. If $\Z_2\times\Z_3\in\Q$, then $\Z_3\in\Q$.
\end{lemma}

\begin{proof}
Consider the span in $\Q$ given by $\langle \iota\colon\Z_2\times\Z_2\hookrightarrow\Z_2\times\Z_3,f\colon\Z_2\times\Z_2\hookrightarrow\Z_2\times\Z_3\rangle$, where $\iota\colon (a,b)\mapsto (a,b)$ is the inclusion embedding and $f\colon (a,b)\mapsto (b,a)$. Since $\Q\in\omq(\SA)$, this span has an amalgam $\langle g_1\colon\Z_2\times\Z_3\hookrightarrow\m{D},g_2\colon\Z_2\times\Z_3\hookrightarrow\m{D}\rangle$, where $\m{D}\in\Q$. Because every Sugihara algebra is isomorphic to a subdirect product of copies of $\Z$, we may assume without loss of generality that $\m{D}$ is a subalgebra of $\Z^\kappa$ for some $\kappa$. We will prove that $\Z_3\in\iso\sub(\m{D})$, from which it will immediately follow that $\Z_3\in\Q$.

Because $|(1,0)|=(1,0)$ in $\Z_2\times\Z_3$ and being positive is preserved by homomorphisms, both $\ut\leq g_1(1,0)$ and $\ut\leq g_2(1,0)$ hold in $\Z^\kappa$, where $\ut\colon\kappa\to\Z$ is the vector in $\Z^\kappa$ that is constantly $0$. We claim that for each $i\in\kappa$ either $g_1(1,0)(i)=0$ or $g_2(1,0)(i)=0$. Toward a contradiction, suppose that $j\in\kappa$ is such that both $g_1(1,0)(j)>0$ and $g_2(1,0)(j)>0$. It follows that $-g_k(1,0)(j)<0$ in $\Z^\kappa$ for each $k\in\{1,2\}$ and, since $g_k$ is a homomorphism for $k\in\{1,2\}$,
\begin{align*}
g_k(1,0)(j)\meet g_k(-1,1)(j) &= g_k(-1,0)(j) \\
&= - g_k(1,0)(j) \\
&< 0.
\end{align*}
It follows that $g_k(-1,1)(j)<0$ since the elements in the $j$th coordinate are linearly ordered and $g_k(1,0)(j)>0$. On the other hand, because $\langle g_1,g_2\rangle$ is an amalgam of $\langle\iota,f\rangle$, we have that $g_1\circ\iota=g_2\circ f$ and so
\[
g_1(-1,1) = g_1(\iota(-1,1)) = g_2(f(-1,1)) = g_2(1,-1) = -g_2(-1,1),
\]
so $g_1(-1,1)(j)<0$ implies $g_2(-1,1)(j)>0$, a contradiction since we have shown that $g_2(-1,1)(j)<0$.

From the above, for every $i\in\kappa$ one of $g_1(1,0)(i)=0$ or $g_2(1,0)(i)=0$ must hold. Thus, for each $i\in\kappa$, $g_1(1,0)(i)\meet g_2(1,0)(i) = 0$, so $g_1(1,0)\meet g_2(1,0) = \ut$. It follows that $\ut\in D$. Since any non-trivial Sugihara algebra containing an involution-fixed element contains a subalgebra isomorphic to $\Z_3$, it follows that $\Z_3\in\iso\sub(\m{D})$ as desired.
\end{proof}

\begin{lemma}\label{lem:4reduction}
Let $\Q\in\omq(\SA)$. If $\Z_2\times\Z_4\in\Q$, then $\Z_4\in\Q$.
\end{lemma}

\begin{proof}
Consider the span $\langle\iota\colon\Z_2\times\Z_2\hookrightarrow\Z_2\times\Z_4,f\colon\Z_2\times\Z_2\hookrightarrow\Z_2\times\Z_4\rangle$, where $\iota$ is the inclusion embedding and $f\colon (a,b)\mapsto (b,a)$. Since $\Q\in\omq(\SA)$, the aforementioned span has an amalgam $\langle g_1\colon\Z_2\times\Z_4\hookrightarrow\m{D},g_2\colon\Z_2\times\Z_4\hookrightarrow\m{D}\rangle$, where $\m{D}\in Q$. We exhibit a subalgebra of $\m{D}$ isomorphic to $\Z_4$, from which it is immediate that $\Z_4\in\Q$.

In $\m{D}$, consider the elements
\begin{align*}
d &= g_1(1,2)\join g_2(1,2)\\
c &= g_1(1,1)\\
b &= g_1(-1,-1)\\
a &= g_1(-1,-2)\meet g_2(-1,-2).
\end{align*}
Since $g_1$ and $g_2$ are isotone embeddings, the elements $a,b,c,d$ are pairwise distinct and, indeed, $a<b<c<d$ in $\m{D}$. Further, because $g_1,g_2$ are homomorphisms, $c=\neg b$ and $d = \neg a$. Thus, $\{a,b,c,d\}$ is closed under $\meet$, $\join$, and $\neg$. We show that $\{a,b,c,d\}$ is closed under multiplication $\cdot$ as well.

It is immediate that $x^2=x$ for each $x\in\{a,b,c,d\}$. Moreover, using the fact that multiplication distributes over both of $\meet$ and $\join$ and that $g_1,g_2$ are homomorphisms, direct computation shows that $ac=a$, $bc=b$, and $cd=d$. Because $\cdot$ is commutative, the only nontrivial products to compute are $ad$ and $bd$. By Lemma~\ref{lem:basic props}(9), $ad=a\cdot\neg a = a\meet\neg a = a$. For the last remaining case, observe that
\begin{align*}
bd &= g_1(-1,-1)[g_1(1,2)\join g_2(1,2)]\\
&= g_1(-1,-1)g_1(1,2)\join g_1(-1,-1)g_2(1,2)\\
&= g_1(-1,-1)g_1(1,2)\join g_2(-1,-1)g_2(1,2)\\
&= g_1(-1,2)\join g_2(-1,2).
\end{align*}
Clearly, $g_1(-1,2)\join g_2(-1,2)\leq d$. On the other hand, since $g_1(x,y)=g_2(y,x)$ for any $x,y\in\{-1,1\}$, we have that
\begin{align*}
g_1(1,1)\join g_2(1,1) &= g_1(-1,1)\join g_1(1,-1)\join g_2(-1,1)\join g_2(1,-1)\\
&= g_1(-1,1)\join g_2(-1,1)\join g_2(-1,1)\join g_1(-1,1)\\
&\leq g_1(-1,2)\join g_2(-1,2),
\end{align*}
so
\[
g_1(-1,2)\join g_2(-1,2) = g_1(-1,2)\join g_2(-1,2) \join g_1(1,1)\join g_2(1,1) = d.
\]
It follows that $\{a,b,c,d\}$ is closed under multiplication and, in fact, $\{a,b,c,d\}$ forms a totally ordered subalgebra of $\m{D}$ that is isomorphic to $\Z_4$. This yields the result.
\end{proof}

\begin{lemma}\label{lem:4lifting}
Let $\Q\in\omq(\SA)$. If $\Z_4\in\Q$, then $\Z_{2n}\in\Q$ for every positive integer $n$. Consequently, if $\Z_4\in\Q$, then $\E\in\Q$.
\end{lemma}

\begin{proof}
We argue by induction on $n$. If $n=2$, the claim is true by the assumption that $\Z_4\in\Q$, giving the base case. Suppose now that $\Z_{2n}\in Q$ for some $n\geq 2$. We show that $\Z_{2n+2}\in\Q$. For this, consider the span $\langle \iota\colon\Z_2\hookrightarrow\Z_4,f\colon\Z_2\hookrightarrow\Z_{2n}\rangle$, where $\iota$ is the identity embedding and $f$ is defined by $f(-1)=-n$ and $f(1)=n$. Since $\Q\in\omq(\SA)$, this span has an amalgam $\langle g_1\colon \Z_4\hookrightarrow\m{D},g_2\colon\Z_{2n}\hookrightarrow\m{D}\rangle$, where $\m{D}\in\Q$.

Set $S=\im(g_1)\cup\im(g_2)$. We claim that $S$ is the universe of a totally ordered subalgebra of $\m{D}$ with exactly $2n+2$ elements, and hence is isomorphic to $\Z_{2n+2}$ by Lemma~\ref{lem:totally ordered}. First, observe that, since $g_1$ and $g_2$ are order embeddings,
\begin{align*}
g_1(-2)<g_1(-1)&=g_2(-n)\\
&<g_2(-n+1)\\
&\;\;\;\vdots\\
&<g_2(n-1)\\
&<g_2(n)=g_1(1)<g_1(2),
\end{align*}
where we have used the equalities $g_1(-1)=g_1(\iota(-1))=g_2(f(-1))=g_2(-n)$ and $g_2(n)=g_2(f(1))=g_1(\iota(1))=g_1(1)$. Thus, the elements of $S$ form an $(2n+2)$-element chain in $\m{D}$. It follows that $S$ is closed under $\meet$ and $\join$, and we claim that it is also closed under $\neg$ and $\cdot$. Closure under $\neg$ follows immediately since $\im(g_2)$ is a subalgebra of $\m{D}$ isomorphic to $\Z_{2n}$ and $\neg g_1(-2) = g_1(2)$.

For closure under $\cdot$, it suffices to show that $g_1(-2)x,g_1(2)x\in S$ for any $x\in\im(g_2)$. Observe that, for any $k\in\Z_{2n}$,
\begin{align*}
g_1(-2) &= g_1(-2)g_1(-1) \\
&= g_1(-2)g_2(-n) \\
&= g_1(-2)g_2(-n)g_2(k) \\
&= g_1(-2)g_2(k),
\end{align*}
where we have used the fact that $g_2(-n)$ is an absorbing element in $\im(g_2)$. That $g_1(2)g_2(k)=g_1(2)$ for any $k\in\Z_{2n}$ follows similarly. Therefore, $S$ is closed under $\cdot$ and, hence, $\Z_{2n+2}\in\iso\sub(\m{D})\subseteq\Q$. It follows by induction that $\Z_{2n}\in\Q$ for each $n\geq 1$, so $\E\in\Q$ as well.
\end{proof}

\begin{lemma}\label{lem:pushup}
Let $\Q\in\omq(\SA)$. If $\Z_3,\Z_4\in\Q$, then $\Q=\SA$.
\end{lemma}

\begin{proof}
We argue by induction to show that $\Z_{2n+1}\in\Q$ for each $n\geq 1$. The base case is true by assumption, so assume that $n\geq 1$ is an integer such that $\Z_{2n+1}\in\Q$. We consider the span $\langle i_1\colon\Z_2\hookrightarrow\Z_4,i_2\colon\Z_2\hookrightarrow\Z_{2n+1}\rangle$, where $i_1$ and $i_2$ are the identity embeddings. By hypothesis, this span has an amalgam $\langle g_1\colon\Z_4\hookrightarrow\m{D},g_2\colon\Z_{2n+1}\hookrightarrow\m{D}\rangle$ for some $\m{D}\in\Q$. Arguing as in the proof of Lemma~\ref{lem:4lifting}, we set $S=\im(g_1)\cup\im(g_2)$ and observe that $S$ is the universe of a subalgebra $\m{S}$ of $\m{D}$ and that
\[
g_1(-2)<g_2(-n)<\cdots<g_2(0)<\cdots<g_2(n)<g_1(2),
\]
so $\m{S}$ is a totally ordered subalgebra of $\m{D}$ with exactly $2n+3$ elements. It follows that $\m{S}\cong\Z_{2n+3}$ by Lemma~\ref{lem:totally ordered}, so $\Z_{2n+3}\in\Q$. Thus, $\Z_{2n+1}\in\Q$ for each $n\geq 1$ and therefore $\Q=\SA$ as desired.
\end{proof}

To complete our main argument, we require one further lemma regarding the structure of $\Lambda(\SA)$. We will use it in tandem with Lemma~\ref{lem:split} to show that the closure properties exhibited in Lemmas~\ref{lem:3reduction}, \ref{lem:4reduction}, \ref{lem:4lifting}, and \ref{lem:pushup} suffice to exclude all nontrivial quasivarieties not already listed in Lemma~\ref{lem:with the ap}.

\begin{lemma}\label{lem: for axiomatization of even chains}
Let $\m{A}$ be a finite Sugihara algebra such that $\m{A}\notin\quas(\E)$. Then either $\Z_3\in\iso\sub( \m{A})$ or $\Z_2\times\Z_3\in\iso\sub (\m{A})$.
\end{lemma}
\begin{proof}
Let $\m{A}\notin\quas(\E)$ be finite.
We proceed inductively on the size of $\m{A}$. It is obvious that any algebra which is not a member of $\quas(\E)$ needs to have at least three elements. Thus, $|\m{A}|=3$ in the base case. This means $\m{A}\in\iso(\Z_3)$, which gives us the result immediately. 

For the inductive step, suppose $|\m{A}|>3$. We have two cases to consider: Either (i) $\m{A}$ is directly indecomposable, or (ii) $\m{A}$ is not directly indecomposable. 
Assume (i). Then, by Lemma \ref{lem: directly indecomposable}, we have $\m{A}\in\iso(\bot\m{B}\top)$, where $|\m{B}|\geq 2$. If $|\m{B}|= 2$, then $\m{A}\in\iso(\Z_4)\subseteq\quas(\E)$, which contradicts the initial assumption that $\m{A}\notin\quas(\E)$. Hence, $|\m{B}|\geq 3$ and so we can apply the inductive hypothesis to $\m{B}$. Thus, either $\Z_3\in\iso\sub( \m{B})$, or $\Z_2\times\Z_3\in\iso\sub (\m{B})$, which immediately gives us the result since $\m{B}\in\iso\sub(\m{A})$.

For the second case, suppose $\m{A}\in\iso(\m{A}_1\times\ldots\times\m{A}_k)$, where $\m{A}_i$ is directly indecomposable for all $i=1,\ldots,k$. If $\m{A}_i\in\iso(\Z_2)$ for all $i=1,\ldots,k$, then $\m{A}\in\quas(\E)$, a contradiction. Thus, there exists $j\leq k$ such that $\m{A}_j\in\iso(\bot\m{B}\top)$ for some finite Sugihara algebra $\m{B}$. Without the loss of generality, we can assume $k$ to be such a $j$. Thus $|\m{A}_k|\geq 3$, so we can apply the inductive hypothesis, obtaining $\Z_3\in\iso\sub(\m{A}_k)$ or $\Z_2\times\Z_3\in\iso\sub(\m{A}_k)$. Furthermore, we have $\Z_2\in\iso\sub(\m{A}_1\times\ldots\times\m{A}_{k-1})$. Since every product of subalgebras is a subalgebra of the product, this yields $\Z_2\times\Z_3\in\iso\sub(\m{A})$ and hence the result.
\end{proof} 

We finally arrive at the main result of this section.

\begin{theorem}\label{lem:AP classification}
The quasivarieties of Sugihara algebras with the amalgamation property are exactly the trivial variety $\vr(\Z_1)$, the variety of Boolean algebras $\vr(\Z_2)$, $\vr(\Z_3)$, $\quas(\E)$, and the variety of all Sugihara algebras $\vr(\Z)$.
\end{theorem}

\begin{proof}
Let $\Q\in\omq(\SA)$ be nontrivial. By Lemma~\ref{lem:with the ap}, it suffices to show that $\Q$ is one of $\vr(\Z_2)$, $\vr(\Z_3)$, $\vr(\Z)$, or $\quas(\E)\}$.

Since $\Q$ is nontrivial, $\vr(\Z_2)=\quas(\Z_2)\subseteq\Q$. Assume that this containment is proper. Then, by Lemma~\ref{lem:split}, either $\Z_2\times\Z_3\in\Q$ or $\Z_2\times\Z_4\in\Q$. We consider three mutually exclusive cases.

First, suppose that both $\Z_2\times\Z_3\in\Q$ and $\Z_2\times\Z_4\in\Q$. Then $\Z_3,\Z_4\in\Q$ by Lemmas~\ref{lem:3reduction} and \ref{lem:4reduction}. It then follows by Lemma~\ref{lem:pushup} that $\Q=\SA=\vr(\Z)$.

Second, suppose that $\Z_2\times\Z_3\in\Q$ and $\Z_2\times\Z_4\not\in\Q$. Then, since $\Z_2\times\Z_3\in\Q$, applying Lemma~\ref{lem:3reduction} gives that $\Z_3\in\Q$. Therefore, $\vr(\Z_3)=\quas(\Z_3)\subseteq\Q$. On the other hand, since $\Z_2\times\Z_4\notin\Q$, Lemma~\ref{lem:splitting for v3} implies that $\vr(\Z_3)$ is not properly contained in $\Q$, i.e., $\Q=\vr(\Z_3)$.

Third, suppose that $\Z_2\times\Z_3\not\in\Q$ and $\Z_2\times\Z_4\in\Q$. Since $\Z_2\times\Z_4\in\Q$, applying Lemmas~\ref{lem:4reduction} and \ref{lem:4lifting} gives that $\E\in\Q$ and, hence, $\quas(\E)\subseteq\Q$. If this containment is proper, then there exists a finite algebra $\m{A}\in\Q$ with $\m{A}\notin\quas(\E)$. Thus, the assumption that $\Z_2\times\Z_3\not\in\Q$ implies, by Lemma~\ref{lem: for axiomatization of even chains}, that $\Z_3$ embeds in $\m{A}$, so $\Z_3\in\Q$. But then $\Q=\SA$ by Lemma~\ref{lem:pushup}, contradicting the assumption that $\Z_2\times\Z_3\not\in\Q$. Thus $\Q=\quas(\E)$, and the result follows.
\end{proof}
The poset of all subquasivarieties of $\SA$ with the \prp{AP} is depicted in Figure~\ref{fig:omegaSA}.

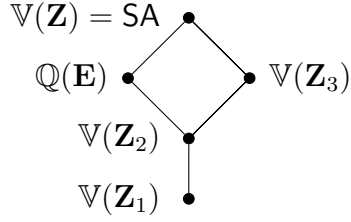
\begin{figure}[t]
\centering
\begin{tikzpicture}[
place/.style={circle,draw=black,fill=black, minimum size = 4pt, inner sep = 0pt},
square/.style={regular polygon,regular polygon sides=4},
place2/.style={square,draw=black,fill=black, minimum size = 5.5pt, inner sep = 0pt},
place3/.style={square,draw=black, minimum size = 5.5pt, inner sep = 0pt},
place4/.style={circle,draw=black, minimum size = 4pt, inner sep = 0pt}]

   \node[place] (tm) at (0,-2.4) {};
  \node[place] (ml) at (-.8,-3.2) {};
    \node[place] (mr) at (0.8,-3.2) {};
  \node[place] (bm) at (0,-4) {};
  \node[place] (bot) at (0,-4.8) {}; 
   
  \node[left] () at (tm) {$\vr(\Z)=\SA$\;\;};
  \node[left] () at (ml) {$\quas(\E)$\;};
  \node[right] () at (mr) {\;$\vr(\Z_3)$};
  \node[left] () at (bm) {$\vr(\Z_2)$\;\;};
  \node[left] () at (bot) {$\vr(\Z_1)$\;\;};
  
  \draw (bot) -- (bm) -- (mr) -- (bm) -- (ml) -- (tm) -- (mr) -- (tm);
\end{tikzpicture}
\caption{The poset $\omq(\SA)$ of quasivarieties of Sugihara algebras with the amalgamation property.}
\label{fig:omegaSA}
\end{figure}

Inspection of the proofs of the closure properties needed for Theorem~\ref{lem:AP classification} attest to the fact that the \prp{RCEP} has been avoided. We thus obtain the following surprising corollary.

\begin{corollary}\label{cor:from AP to RCEP}
Let $\Q$ be any quasivariety of Sugihara algebras. If $\Q$ has the amalgamation property, then $\Q$ has the relative congruence extension property.
\end{corollary}

\begin{proof}
The quasivarieties of Sugihara algebras are listed in Lemma~\ref{lem:RCEP for Q SA}, and directly comparing this with the list of quasivarieties with \prp{AP} in Theorem~\ref{lem:AP classification} gives the result.
\end{proof}

Recall from Lemma~\ref{lem:tip} that a quasivariety $\Q$ has the \prp{TIP} if and only if $\Q$ has both the \prp{RCEP} and \prp{AP}. Together with Corollary~\ref{cor:from AP to RCEP}, this supplies the following result.

\begin{corollary}\label{cor:from AP to TIP}
Let $\Q$ be any quasivariety of Sugihara algebras. Then $\Q$ has the amalgamation property if and only if $\Q$ has the transferable injections property.
\end{corollary}

\section{From amalgamation to interpolation}
\label{sec:exclusion}

Recall that a logic $\langle\mathcal{L},\vdash\rangle$ has the \emph{Maehara interpolation property} (or \emph{\prp{MIP}}) if for any set of formulas $\Sigma\cup\Gamma\cup\{\alpha\}$, if $\var(\Sigma\cup\{\alpha\})\cap\var(\Gamma)\neq\emptyset$ and $\Sigma,\Gamma\vdash\alpha$, there exists a set of formulas $\Delta$ such that $\var(\Delta)\subseteq\var(\Sigma\cup\{\alpha\})\cap\var(\Gamma)$, $\Gamma\vdash\Delta$, and $\Sigma,\Delta\vdash\alpha$.

As a consequence of Theorem~\ref{lem:AP classification} and Corollary~\ref{cor:from AP to TIP}, we have a complete description of the subquasivarieties of $\SA$ with the \prp{TIP}. On a logical level, the significance of this is encapsulated in the following lemma. 

\begin{lemma}[{see \cite[Theorem~2.2]{Czel1985}}]\label{lem:bridge}
Let $\langle\mathcal{L},\vdash\rangle$ be an algebraizable logic whose equivalent algebraic semantics is the quasivariety $\Q$. Then $\langle\mathcal{L},\vdash\rangle$ has the Maehara interpolation property if and only if $\Q$ has the transferable injections property.
\end{lemma}

We consider another logical property. Let $\lgc{L}=\langle \mathcal{L},\vdash\rangle$ be any logic. A subset $T\subseteq\fm_\mathcal{L}(X)$ is called a \emph{theory of \lgc{L} over the set of variables $X$} provided that $T\vdash\alpha$ implies $\alpha\in T$, for any $\alpha\in\fm_\mathcal{L}(X)$. The logic $\lgc{L}$ has the \emph{Robinson property} (or \emph{\prp{RP}} for short) when it satisfies the following condition:
\begin{equation}
  \tag{\prp{RP}}\label{eq:RP}
  \parbox{\dimexpr\linewidth-6em}{%
    \strut
	Whenever $X,Y$ are sets of variables such that $X\cap Y\neq\emptyset$, $T$ is a theory of $\lgc{L}$ over $X$, and $S$ is a theory of $\lgc{L}$ over $Y$ such that $T\cap\fm_\mathcal{L}(X\cap Y)=S\cap\fm_\mathcal{L}(X\cap Y)$, there exists a theory $R$ of $\lgc{L}$ over $X\cup Y$ such that $T=R\cap\fm_\mathcal{L}(X)$ and $S=R\cap\fm_\mathcal{L}(Y)$.
    \strut
    }
\end{equation}
The following well-known result links the \prp{RP} to amalgamation.

\begin{lemma}[{\cite[Corollary~5.28]{CzelPig1999}}]\label{lem:bridge2}
Let $\langle\mathcal{L},\vdash\rangle$ be an algebraizable logic whose equivalent algebraic semantics is the quasivariety $\Q$. Then ${\langle\mathcal{L},\vdash\rangle}$ has the Robinson property if and only if $\Q$ has the amalgamation property.
\end{lemma}

Using Lemmas~\ref{lem:bridge} and \ref{lem:bridge2}, the following is immediate from Corollary~\ref{cor:from AP to TIP}. Here we make use of the well-known fact that $\lgc{RM}$ has $\SA$ as its equivalent algebraic semantics.

\begin{proposition}\label{prop:RP iff MIP}
Let $\lgc{L}$ be any extension of $\lgc{RM}$. Then $\lgc{L}$ has the Robinson property if and only if $\lgc{L}$ has the Maehara interpolation property.
\end{proposition}

We now arrive at the last of our main theorems. In order to state in properly, we recall the \emph{Dugundji formulas}, which provide axiomatizations for the extensions of $\lgc{RM}$ corresponding to $\vr(\Z_n)$, $n\geq 1$. Let $p_1,p_2,p_3\ldots$ be distinct propositional variables. We define $\delta_1$ to be equal to the propositional variable $p_1$, and for $n\geq 2$ we define $\delta_n$ to be the disjunction
\[
\bigvee_{1\leq i < j\leq n} (p_i\leftrightarrow p_j),
\]
where $\alpha\leftrightarrow\beta$ abbreviates the formula $(\alpha\to\beta)\meet (\beta\to\alpha)$. It is proven in \cite[Corollary~2]{Dunn1970} that the logic corresponding to $\vr(\Z_n)$ is axiomatized relative to $\lgc{RM}$ by the single formula $\delta_n$.

\begin{theorem}\label{thm:MIP classification}
There are exactly five extensions of $\lgc{RM}$ with the Maehara interpolation property. These are:
\begin{enumerate}
\item $\lgc{RM}$ itself.
\item The trivial logic, obtained from $\lgc{RM}$ by adding the axiom $\delta_1$.
\item Classical propositional logic, obtained from $\lgc{RM}$ by adding the axiom $\delta_2$.
\item The extension of $\lgc{RM}$ by the axiom $\delta_3$.
\item The extension of $\lgc{RM}$ by the rules $\alpha,\neg\alpha\vdash\beta$ and $\alpha,\neg\alpha\join\beta\vdash\beta$.
\end{enumerate}
\end{theorem}

\begin{proof}
By Lemma~\ref{lem:bridge} and the fact that $\lgc{RM}$ is algebraized by $\SA$, an extension of $\lgc{RM}$ has the \prp{MIP} if and only if the corresponding subvariety of $\SA$ has the \prp{TIP}. Thus, by Theorem~\ref{lem:AP classification} and Corollary~\ref{cor:from AP to TIP}, the extensions of $\lgc{RM}$ with the \prp{MIP} are exactly the extensions corresponding to the quasivarieties $\vr(\Z_1)$, $\vr(\Z_2)$, $\vr(\Z_3)$, $\quas(\E)$, and $\vr(\Z)=\SA$ The extension corresponding to $\SA$ is just $\lgc{RM}$ itself, and the extensions correspond to $\vr(\Z_1)$, $\vr(\Z_2)$, $\vr(\Z_3)$ are respectively the extensions by the Dugundji formulas $\delta_1$, $\delta_2$, and $\delta_3$, respectively.

All that remains to show is that the extension $\vdash_\mathrm{E}$ corresponding to $\quas(\E)$ may be obtained by adding the rules $\alpha,\neg\alpha\vdash\beta$ and $\alpha,\neg\alpha\join\beta\vdash\beta$ to $\lgc{RM}$. Let $\vdash'$ be the extension of $\lgc{RM}$ by adding these two rules. It is routine to check that each of these rules is valid in $\E$, and hence $\alpha,\neg\alpha\vdash_\mathrm{E}\beta$ and $\alpha,\neg\alpha\join\beta\vdash_\mathrm{E}\beta$, so $\vdash'\subseteq\vdash_\mathrm{E}$. On the other hand, let $\Q$ be the subquasivariety of $\SA$ corresponding to $\vdash'$. Then $\quas(\E)\subseteq\Q$. If this inclusion is proper, then, by Lemma~\ref{lem: for axiomatization of even chains}, either $\Z_3\in\Q$ or $\Z_2\times\Z_3\in\Q$. But this is a contradiction, since $\Z_3$ refutes the rule $\alpha,\neg\alpha\vdash\beta$ and $\Z_2\times\Z_3$ refutes the rule $\alpha,\neg\alpha\join\beta\vdash\beta$. It follows that $\Q=\quas(\E)$, so the result has been proven.
\end{proof}

We say that an extension $\vdash$ of $\lgc{RM}$ is \emph{finitely based} if there exists a finite collection of rules $\Sigma$ such that $\vdash$ is the extension of $\lgc{RM}$ by $\Sigma$. Clearly, an extension of $\lgc{RM}$ is finitely based if and only if the quasivariety of Sugihara algebras comprising its equivalent algebraic semantics is finitely based as a quasivariety.

\begin{theorem}\label{thm:decidability}
\begin{enumerate}
\item[]
\item It is effectively decidable whether a finitely based subquasivariety of $\SA$ has the \prp{AP}.
\item it is effectively decidable whether a finitely based extension of $\lgc{RM}$ has the \prp{MIP}.
\end{enumerate}
\end{theorem}

\begin{proof}
We prove that if $\Q$ is a finitely based, locally finite quasivariety of finite type and $\Q_1,\Q_2\subseteq\Q$ are subquasivarieties defined relative to $\Q$ by finite sets of quasiequations, then it is effectively decidable whether $\Q_1\subseteq\Q_2$. This suffices to prove the theorem since, as a consequence, it is effectively decidable whether two quasivarieties defined relative to $\SA$ by finite sets of quasiequations coincide and, in particular, it is effectively decidable whether a given quasivariety $\Q$ defined relative to $\SA$ by finitely many quasiequations is one of the five quasivarieties with the \prp{AP}, cf.~Theorem~\ref{thm:MIP classification}.

It is well known that any finitely based, locally finite quasivariety of finite type has a decidable quasiequational theory; see, e.g., \cite[Lemma~6.40]{GalatosJipsenKowalskiOno2007} and \cite[pp. 44-45]{StJohnThesis}. So, suppose that $\Q$ is a finitely based, locally finite quasivariety defined by finitely many quasiequations, and let $\Q_1,\Q_2$ by subquasivarieties of $\Q$ defined, respectively, by the finite sets of quasiequations $\Sigma_1$ and $\Sigma_2$. Because $\Q_2$ is locally finite, its quasiequational theory is decidable. Hence, it is decidable whether $\Q_2\models\Sigma_1$. But $\Q_2\models\Sigma_1$ if and only if $\Q_1\subseteq\Q_2$, so this is to say that it is decidable whether $\Q_1\subseteq\Q_2$. The result follows.
\end{proof}

As a concluding remark, we note that our description of $\omq(\SA)$ yields only a characterization of the extensions of $\lgc{RM}$ with the \prp{MIP}, not necessarily of all extensions with the weaker \prp{DIP}. In the presence of truth constants, the \prp{RP} and \prp{DIP} are known to coincide for substructural logics with a local deduction theorem (see \cite{KiharaOno2010}), as well as with the \prp{MIP}. However, in the absence of a local deduction theorem, the relationship between \prp{RP} and \prp{DIP} is not known. There are conceivably many extensions of $\lgc{RM}$ without local deduction theorems---i.e., corresponding to quasivarieties of Sugihara algebras without \prp{RCEP}---that nevertheless have \prp{DIP}. Pinning down the precise relationship between these metalogical properties is the subject of much on-going work, and we believe that the present study plays a role in the eventual solution of this mystery.

\end{document}